\documentclass[10pt]{article}
\usepackage{amsfonts}
\usepackage{amsmath}
\usepackage{amssymb}
\usepackage{empheq}
\usepackage{latexsym}
\usepackage{amscd}
\usepackage{bbm}
\usepackage{url}
\usepackage[left=3.4cm,right=3.4cm, top=2.5cm,bottom=2.5cm,bindingoffset=0cm]{geometry}

\newcommand\addressT{\noindent\leavevmode
\medskip
\noindent
Kateryna Tatarko, \\
Dept.~of Math.~and Stat.~Sciences,\\
University of Alberta, \\
Edmonton, Alberta, Canada, T6G 2G1.\\
\texttt{\small
e-mail:  tatarko@ualberta.ca}
}

\newcommand\addressW{\noindent\leavevmode
	\medskip
	\noindent
	Elisabeth Werner, \\
	Department of Mathematics,\\
	Case Western Reserve University, \\
	Cleveland, Ohio, USA, 44106.\\
	\texttt{\small
		e-mail: elisabeth.werner@case.edu }
}
\date{}

\newcounter{theorem}[section]
\newtheorem{theorem}{Theorem}

\newtheorem{cor}[theorem]{Corollary}

\newtheorem{theoremA}{Theorem A}

\newcommand{\R}{\mathbb{R}}

\def\vol{{\rm vol}}

\begin{document}

\title{A   Steiner formula in the $L_p$ Brunn Minkowski theory
}

\author{Kateryna Tatarko and Elisabeth M. Werner
\thanks{Partially supported by  NSF grant DMS-1811146}}

\date{}

\maketitle

\begin{abstract}
	We prove an analogue of the classical Steiner  formula for the $L_p$ affine surface area  of a Minkowski outer parallel body for any real parameters $p$.
	We show that  the classical Steiner formula and the Steiner formula of Lutwak's dual Brunn Minkowski theory  are   special cases of this new Steiner formula.
	This  new Steiner formula  and its localized versions lead to new curvature measures that have not appeared before in the literature. 
	They have  the intrinsic volumes of the classical Brunn Minkowski  theory and the  dual quermassintegrals of the  dual Brunn Minkowski theory as well 
	as special cases.
	\par
	Properties of these new quantities are investigated, a connection to information theory among them. A Steiner formula for the  $s$-th mixed $L_p$ affine surface area of a Minkowski outer parallel body for any real parameters $p$ and~$s$ is also given.

	%
	%

\end{abstract}
\vskip 3mm

\section{Introduction and Results}
\subsection {Introduction}
The Brunn Minkowski theory whose classical theory is  also  called the theory of  
{\em mixed volumes} is the very core of convex geometric analysis.
It centers around the study of geometric invariants and geometric measures associated with convex bodies.  A central part of the theory is the classical Steiner formula.
It is  via this formula that the  
elementary mixed volumes, or intrinsic volumes,   are defined
as the coefficients  in the polynomial expansion in $t$ of the volume of the outer parallel body $K+ t B^n_2$ of the convex body $K$ with the Euclidean unit ball $B^n_2$
\begin{equation}\label{Steiner polynomial}
\vol_n(K+tB^n_2) = \sum\limits_{i = 0}^n {n \choose i} W_i(K)t^i = \sum\limits_{i = 0}^n  \vol_{n-i} (B^{n-i}_2) \  V_i(K)t^{n - i}.
\end{equation}
 The coefficients $W_i(K)$ and $V_i(K) $ are called quermassintegrals and intrinsic volumes, respectively. 
Note that the intrinsic volumes $V_i(K) = {n \choose i } \frac{W_{n-i}(K)}{\vol_{n-i} (B^{n-i}_2)} $ are just  differently normalized  quermassintegrals. 
Intrinsic volumes, respectively quermassintegrals,  are the principal geometric invariants in the Brunn Minkowski theory and include volume, surface area and mean width.
Results from the classical Brunn Minkowski theory have been applied to and have connections with numerous mathematical disciplines.
\vskip 2mm
A theory analogous and dual to the  Brunn Minkowski theory, called the theory of {\it dual mixed volumes} or {\it dual Brunn Minkowski theory},   was introduced by Lutwak \cite{Lu1}.
The main geometric invariants in the dual Brunn Minkowski theory are the dual quermassintegrals $\widetilde{W}_i(K)$. As proved in \cite{Lu1}, they appear as the 
coefficients in the Steiner formula of the dual Brunn Minkowski theory
where Minkowski addition ``$+$" of convex bodies is replaced by the radial addition ``$\widetilde{+}$'' of star bodies. Indeed, for a star body $K$ and $t>0,$ we have that
\begin{equation*}
\vol_n(K\widetilde{+} tB^n_2) = \sum\limits_{i = 0}^n {n \choose i}\   \widetilde{W}_i(K)t^i = \sum\limits_{i = 0}^n  {n \choose i} \   \omega_{n - i} \widetilde{V}_i(K)t^{n - i}.
\end{equation*}
The details are given in  Section \ref{BackDG}. Investigations in the dual Brunn Minkowski theory led to isoperimetric inqualities  and kinematic formulas involving dual mixed volumes,
as summarized in \cite{Gardner, SchneiderBook}.
\vskip 2mm
\noindent
A localization of the  quermassintegrals  gives rise to curvature measures for Borel sets in $\mathbb{R}^n$,  respectively 
area measures for Borel sets on the Euclidean unit sphere.  
The {\em classical  Minkowski problem} asks to characterize those measures. 
Similarly, a localization of  the   dual quermassintegrals leads to the dual curvature measures \cite{Lu1}, and analogously to the 
{\em dual  Minkowski problem}. Much work has been devoted to these problems. 
We refer to, e.g.,   \cite{BHP, BLYZ1, BLYZ2, ChenYau, HuangLYZ, Pogorelov, SA2,  Zhao} for background and progress.
\vskip 2mm
An extension of the classical Brunn Minkowski theory,  the {\it $L_p$ Brunn Minkowski theory}, was  initiated by Lutwak in the ground
breaking paper \cite{Lu3}. 
This theory  evolved rapidly over the last years and
due to a number of highly influential works, see, e.g., 
\cite{GaHugWeil, Hab, HabSch2, HLYZ, Lud3, LR2},  
\cite{MeyerWerner} -  \cite{PW1}, 
\cite{SW}  -  
\cite{WY2010}, 
it  is now  a central part of modern convex geometry. 
The $L_p$ Brunn Minkowski theory centers around the study of affine invariants associated with convex bodies.  In fact,   this theory
redirected much of the research about  convex bodies from  the Euclidean aspects 
to the study of the affine geometry of these bodies,  and some questions that had been considered Euclidean in nature turned out to be affine problems. For example,  the famous Busemann-Petty Problem (finally laid to rest in
\cite{Ga1, GaKoSch,  Z1, Z2}), was shown to be an affine problem with the introduction of intersection bodies by Lutwak in \cite{Lu1988}.
The central objects  in  the $L_p$ Brunn Minkowski theory are the $L_p$ affine surface areas,
\begin{equation*} 
as_{p}(K)=\int_{\partial K}\frac{H_{n-1}(x)^{\frac{p}{n+p}}}
{\langle x,\nu(x)\rangle ^{\frac{n(p-1)}{n+p}}} d\mathcal{H}^{n-1}(x),
\end{equation*}
where $\nu(x)$ denotes the outer unit normal at $x \in \partial K$, the boundary of $K$, $H_{n-1}(x)$ is the Gauss curvature at $x$ and $\mathcal{H}^{n-1}$ is the usual surface area measure on $\partial K$.
\vskip 3mm
In this paper we provide  an analogue of the Steiner  formula for the $L_p$ affine surface area and $s$-th mixed $p$-affine surface area  (see Section \ref{MAS})
of a Minkowski outer parallel body for any real parameters $p \ne -n$ and~$s$, i.e. we investigate 
$$
as_p(K + t B^n_2)  \hskip 3mm \text{and}  \hskip 3mm as_{p,s} (K + t B^n_2).
$$
A different Steiner formula for the $L_p$ Brunn Minkowski theory, involving Blaschke sum instead of Minkowski sum,  was put forward by Lutwak in \cite{Lu}.
\par
\noindent
Our new Steiner formula, presented in Theorem A and \ref{MT1}, covers a whole  range of Steiner formulas  for all  $-\infty \leq p \leq \infty$. It  includes the classical Steiner formula and the Steiner formula from the dual $L_p$ Brunn Minkowski theory as special cases.
\par
We call the coefficients in  our Steiner formula  {\em $L_p$ quermassintegrals}. In their general form they do not seem to have appeared before in the literature.
Special cases of those include not only the classical quermassintegrals and the dual quermassintegrals,
but also variants of the Willmore energy.  This is explained in Section \ref{Coeff}. We also observe a connection of the $L_p$ quermassintegrals  to information theory.
Analogously to the curvature measures of the classical theory and the dual curvature measures of the dual Brunn Minkowski theory, 
the $L_p$ quermassintegrals  lead to new  curvature measures and area measures, respectively. 
They are discussed in more detail in Section \ref{Steiner}. The Steiner formulas for the $s$-th mixed $p$ affine surface area is treated in Section \ref{MAS}.


\subsection{Results}
For a convex body $K,$ we define
\begin{equation}\label{minhk}
\beta(K) = \min\limits_{u\in S^{n - 1}} h_K(u).
\end{equation}
As $K$ is such that its centroid is at the origin, there is $\lambda(K) > 0$ such that $B^n(0, \lambda(K) )  \subset K$. As $K$ is bounded, there is $\lambda(K) \leq \Lambda(K) <\infty$
such that $K \subset B^n(0, \Lambda(K)) $. Thus, 
$$
\lambda(K) \leq \beta(K) \leq \Lambda(K).
$$
\noindent

\begin{theoremA}\label{MT11}
	Let $K$ be a convex body in $\mathbb{R}^n$ that is $C^2_+$ and let $t \in \mathbb {R}$ be such that $ 0 \leq t  < \beta(K)$.  For all $p \in \mathbb{R}$, $p \neq -n$, 
	\begin{eqnarray}  \label{formulaA}
	&& as_{p}(K + tB^n_2) = \nonumber \\
	&&\sum\limits_{m = 0}^{\infty} \Biggl[\sum\limits_{k = m}^{\infty} {\frac{n(1-p)}{n+p} \choose k - m}\, t^k \int\limits_{S^{n - 1}} f_K(u)^{\frac n{n + p}} h_K(u)^{\frac{n(1-p)}{n+p}-k + m} A^m_p\, d\sigma(u)\Biggr].
	\end{eqnarray}
\end{theoremA}
When integrating over the sphere, we write $\sigma(u) = \mathcal{H}^{n-1}(u)$. The coefficients $A^m_p$ represent a sum of mixed products of the elementary symmetric functions of the principal curvatures $H_i=H_i \left(\bar{\xi}_K(u) \right)$  (see Section \ref{BackDG} below),   with corresponding multinomial coefficients. The  detailed statement and local versions of it will be given in Section 3. 
\par
\noindent
We call the coefficients in the  formula (\ref{formulaA})  the $L_p$ quermassintegrals
\begin{equation} \label{LpQMI}
\mathcal{W}_{m,k} (K)= \int\limits_{S^{n - 1}} f_K(u)^{\frac n{n + p}} h_K(u)^{\frac{n(1-p)}{n+p}-k + m} A^m_p\, d\sigma(u).
\end{equation} 
The expression $f_K(u)h_K(u)^{1-p}$ is called the {\em $L_p$ curvature function} $f_p(K, \cdot) : S^{n-1} \rightarrow [0, \infty)$. With this notation we can write 
\begin{equation*}
\mathcal{W}_{m,k} (K) = \int\limits_{S^{n - 1}} f_p(K, u)^{\frac n{n + p}} h_K(u)^ {m-k}  \  A^m_p\, d\sigma(u).
\end{equation*} 

We want to point out that  the first coefficient in the expansion (\ref{formulaA}) represents the $L_p$ affine surface area of the body $K$,
$$ 
\mathcal{W}_{0,0}(K) = as_p(K).
$$
In  other special cases, we get mixed affine surface areas, defined in Section \ref{MAS}.
More properties of the  coefficients  are discussed in Section \ref{Coeff}, a connection to information theory among them.
\vskip 2mm
\noindent
The case $p=0$ of Theorem A reduces to the classical Steiner formula which involve the quermassintegrals $W_i(K)$  (\ref{QMI}).
\begin{cor}\label{C1}
	Let $K$ be a convex body in $\mathbb{R}^n$ that is $C^2_+$. Then
	\begin{eqnarray*} \label{p=0-formula}
		as_{0}(K + tB^n_2) = \sum\limits_{i = 0}^n {n \choose i} W_{i}(K) \  t^i.
	\end{eqnarray*}
\end{cor}
\vskip 2mm
In the case $p=\pm \infty$,  Theorem A reduces to the   Steiner formula of the dual Brunn Minkowski theory   involving  Lutwak's  dual quermassintegrals $\widetilde{W}_i(K)$  (\ref{dualQMI}).
\begin{cor}\label{C2}
	Let $K$ be a convex body in $\mathbb{R}^n$ that is $C^2_+$. Let $t \in \mathbb {R}$ be such that $ 0 \leq t  < \beta(K)$. Then
	\begin{eqnarray*} \label{p=0-formula}
		as_{\pm \infty }(K + tB^n_2) &=& n\   \vol((K + tB^n_2)^\circ) = n \  \widetilde{W}_{0} ((K + tB^n_2)^\circ) \\
		&=& n\  \sum\limits_{i = 0}^\infty {-n \choose i} \widetilde{W}_{-i}(K^\circ)\  t^i.
	\end{eqnarray*}
\end{cor}
Here, $K^\circ=\{y \in \mathbb{R}^n: \langle x, y \rangle  \  \  \text{for all  }  x\in K\}$ is the polar body of a convex body $K$.
\vskip 3mm
Thus Theorem A covers a whole  range of Steiner formulas in the $L_p$ Brunn Minkowski theory for all  $-\infty \leq p \leq \infty$, including the classical case and a case related to the  dual Brunn Minkowski theory.

\section{Background} 

\subsection{Background from differential geometry} \label {BackDG}

For more information and the details in this section we refer to e.g., \cite{Gardner,  SchneiderBook}.
\vskip 2mm
Let $K$ be a convex body of class $C^2.$ For a point $x$ on the boundary $\partial K$ of $K$  we denote by $\nu(x)$ the unique   outward unit normal vector of $K$ at $x.$ The map $\nu: \partial K \to S^{n - 1}$ is called the spherical image map or Gauss map of $K$ and is of class $C^1.$ Its differential is called the Weingarten map. The eigenvalues of Weingarten map are the principal curvatures $k_i(x)$ of $K$ at $x.$

The $j$-th normalized elementary symmetric functions of the principal curvatures are denoted by $H_j$. They are defined as follows
\begin{equation}\label{ESFPC}
H_j = {n - 1 \choose j}^{-1} \sum_{1\leq i_1 < \dots < i_j \leq n - 1} k_{i_1} \cdots k_{i_j}
\end{equation}
for $j = 1, \dots, n-1$ and $H_0 = 1.$ Note that 
$$H_1 = \frac{1}{n-1}  \sum_{1\leq i \leq n - 1} k_{i} $$ is the mean curvature, that is the average of principal curvatures, and 
$$H_{n - 1}= \prod_{i=1}^{n-1} k_i$$ 
is the Gauss curvature. 
\par
We say that $K$ is of class $C^2_+$ if $K$ is of class $C^2$ and the Gauss map $\nu$ is a diffeomorphism. This means in particular that  $\nu$ has a smooth inverse. This assumption is stronger than just $C^2,$ and  is equivalent to the assumption that all principal curvatures are strictly positive, or that the Gauss curvature $H_{n-1} \ne 0.$
It also means that the differential of $\nu$, i.e., the Weingarten map, is of  maximal rank everywhere. 
\par
Let $K$ be of class $C^2_+$. For $u \in \R^n \setminus\{0\}$, let $\xi_K(u)$ be the unique point on the boundary of $K$ at which $u$ is an outward  normal vector. 
The map $\xi_K$ is defined on $\R^n \setminus \{0\}$. Its restriction to the sphere $S^{n - 1},$ the map $\bar{\xi}_K: S^{n - 1} \to \partial K$, is called the reverse spherical image map, or reverse Gauss map.  The differential of $\bar{\xi}_K$ is called the reverse Weingarten map. The eigenvalues of  the reverse Weingarten map are called the principal radii of curvature $r_1, \dots, r_{n - 1}$ of $K$ at $u\in S^{n - 1}.$ 
\par
The $j$-th normalized elementary symmetric functions of the principal radii of curvature are denoted by $s_j$. In particular, $s_0=1$,   and for $1 \leq j \leq n-1$ they are defined
\begin{equation}\label{ESFPR}
s_j = {n - 1 \choose j}^{-1} \sum_{1\leq i_1 < \dots < i_j \leq n - 1} r_{i_1} \cdots r_{i_j}.
\end{equation}
Note that the principal curvatures are functions on the boundary of $K$ and the principal radii of curvature are functions on the sphere. 
\par
Now we describe  the connection between $H_j$ and $s_j.$ For a body $K$ of class $C^2_+$, we have for $u\in S^{n - 1}$ that $\bar{\xi}_K(u) = \nu^{-1}(u)$. 
In particular, the principal radii of curvature are reciprocals of the principal curvatures, that is
$$
r_i(u) = \frac1{k_i(\bar{\xi}_K(u))}.
$$
This implies that for $x \in \partial K$ with $\nu(x)=u$, 
$$
s_j = {n - 1 \choose j}^{-1} \sum_{1\leq i_1 < \dots < i_j \leq n - 1} \frac1{k_{i_1}(\bar{\xi}_K(u)) \cdots k_{i_j}(\bar{\xi}_K(u))} = \frac{H_{n - 1 -j}}{H_{n - 1}}\Big(\bar{\xi}_K(u)\Big)
$$
and
$$
H_j  =  \frac{s_{n - 1 -j}}{s_{n - 1}}\Big(\nu(x)\Big), 
$$
for $j = 1, \dots, n-1.$
\par
The mixed volumes $W_i(K)$ of the classical Steiner formula (\ref{Steiner polynomial}) can be expressed with the elementary symmetric functions of the principal curvatures.  
By definition, $W_0(K) = \vol_n(K)$ and
\begin{equation}\label{QMI}
W_i(K) = \frac1n \int\limits_{\partial K} H_{i - 1} d\mathcal{H}^{n - 1}, 
\end{equation} 
for $i = 1, \dots, n.$
Moreover, we have  the following formulae, which in integral geometry and  differential geometry are known as {\it Minkowskian integral formulae}
\begin{equation*}
\int\limits_{S^{n - 1}} s_j d\mathcal{H}^{n-1} = \int\limits_{S^{n - 1}} h_K s_{j - 1} d\mathcal{H}^{n - 1},
\end{equation*}
and
\begin{equation} \label{mean curvature functions relation}
\int\limits_{\partial K} H_{j - 1} d\mathcal{H}^{n-1} = \int\limits_{\partial K} \langle x, \nu(x) \rangle H_j d\mathcal{H}^{n - 1}
\end{equation}
for $j = 1, \dots, n-1$.
\vskip 2mm
\noindent
The dual mixed volume of the convex bodies $K$ and $L$ that contain $0$ in its interior is defined for all real  $i$  by
$$
\widetilde{V}_i (K, L) = \frac1n \  \int\limits_{S^{n - 1}} \rho_K(u)^{n - i} \rho_L(u)^i d\sigma(u), 
$$
where $\rho_K(u) = \max\{\lambda \geq 0 |\, \lambda u \in K\}$ is the radial function of $K.$ In particular, if $L = B^n_2$ then
\begin{equation}\label{dualQMI}
\widetilde{W}_i(K) = \widetilde{V}_i (K, B^n_2) = \frac{1}{n} \int\limits_{S^{n - 1}} \rho_K(u)^{n - i} d\sigma(u)
\end{equation}
are called {\it dual quermassintegrals} of order $i.$ Then
the corresponding Steiner  formula in the dual Brunn Minkowski theory is
\begin{equation*}
\vol_n(K\widetilde{+} tB^n_2) = \sum\limits_{i = 0}^n {n \choose i} \widetilde{W}_i(K) \  t^i .
\end{equation*}
\par
A well  known change of integral formula, e.g., \cite{SchneiderBook}, which we use frequently, says that  for a $C^2_+$ convex body $K$ and a continuous function  $g:\partial K \to\mathbb {R}$
\begin{equation} \label{Integration}
\int\limits_{S^{n - 1}} g(u) f_K(u)  d\sigma (u)= \int\limits_{\partial K} g(x) d\mathcal{H}^{n-1}(x) ,
\end{equation}
where $u \in S^{n - 1}$  and $x \in \partial K$ are related via the Gauss map, i.e.,  $\nu(x) = u$.


\subsection{Background from affine geometry} \label{BackAff}

From now on we will always assume that  the centroid of a
convex body $K$ in $\mathbb R^n$ is at the origin. For real  $p \neq -n$, we define  the
$L_p$ affine surface area $as_{p}(K)$ of $K$ as in \cite{Lu1} ($p
>1$) and \cite{SW1} ($p <1,\ p \ne -n$) by
\begin{equation} \label{def:paffine}
as_{p}(K)=\int\limits_{\partial K}\frac{H_{n-1}(x)^{\frac{p}{n+p}}}
{\langle x,\nu(x)\rangle ^{\frac{n(p-1)}{n+p}}} \, d\mathcal{H}^{n-1}(x)
\end{equation}
and
\begin{equation}\label{def:infty}
as_{\pm\infty}(K)=\int\limits_{\partial K}\frac{H_{n-1} (x)}{\langle
	x,\nu(x)\rangle ^{n}} d\mathcal{H}^{n-1}(x).
\end{equation}
In particular, for $p=0$, 
\begin{equation}\label{0-asa}
as_{0}(K)=\int\limits_{\partial K} \langle x,\nu(x)\rangle
\,  d\mathcal{H}^{n-1}(x) = n\, \vol(K).
\end{equation}
The case $p=1$, 
$$
as_{1}(K)=\int\limits_{\partial K}H_{n-1}(x)^{\frac{1}{n+1}}  d\mathcal{H}^{n-1}(x)
$$
is the classical affine surface area  which is independent
of the position of $K$ in space.  For dimensions 2 and 3 and sufficiently smooth convex bodies, its definition goes 
back to Blaschke \cite{Blaschke}.
\par
\noindent
If the boundary of $K$ is sufficiently smooth then (\ref{def:paffine})
and (\ref{def:infty}) can be written as integrals over the boundary
$\partial B^n_2=S^{n-1}$ of the Euclidean unit ball $B^n_2$  in $\mathbb R^n$,
\begin{equation} \label{pasa}
as_{p}(K)=\int\limits_{S^{n-1}}\frac{f_{K}(u)^{\frac{n}{n+p}}}
{h_K(u)^{\frac{n(p-1)}{n+p}}}
d\mathcal{H}^{n-1}(u),
\end{equation}
where $f_{K}(u)$ is the curvature function, i.e. the reciprocal of
the Gaussian curvature $H_{n-1}(x)$ at this point $x \in
\partial K$ that has $u$ as outer normal. In particular, for
$p=\pm \infty$,
\begin{equation}\label{inf-aff}
as_{\pm\infty}(K)
=\int\limits_{S^{n-1}}\frac{1}{h_K(u)^{n}}
d\sigma(u)
=n \, \vol(K^{\circ}), 
\end{equation}
where $K^\circ=\{y\in \mathbb{R}^n, \langle x, y\rangle\leq 1, \forall
x\in K\}$ is the polar body of $K$. 
\vskip 2mm
\noindent
For $p=-n$, the $L_{-n}$-affine surface area was introduced in \cite{MeyerWerner} as
\begin{equation}\label{L-n}
as_{-n}(K) = \max _{u \in S^{n-1}} f_K(u)^\frac{1}{2} h_K(u)^\frac{n+1}{2}.
\end{equation}


\section{The Steiner formula of the $L_p$ Brunn Minkowski theory} \label{Steiner}

In this section, we state  our main theorems  and discuss some of their consequences.  The proofs are given in Section \ref{proofs}.   
\subsection{The general case}
\par
\noindent
We will need the  generalized binomial coefficients. For $\alpha \in \mathbb{R}$ and $k \in \mathbb{N}$, they are defined as
\begin{equation} \label{gcoef}
{\alpha \choose k} =
\left\{
\begin{aligned}
&1 \quad &\text{ if }\  k=0,\\
&0 \quad &\text{ if }\  k<0,\\
&\frac{\alpha(\alpha - 1)\cdots(\alpha - k + 1)}{k!} \quad  &\text{ if } k > 0.
\end{aligned}\right.
\end{equation}
\par
\noindent
Also, we will need the multinomial coefficients  from the multinomial formula
\begin{equation}\label{multiformula}
(a_1 + \ldots +a_r)^q = \sum_{\substack{
		i_1, \dots, i_r \geq 0\\
		i_1+\dots+i_r = q}} 
{q \choose i_1,\dots, i_r}a_1^{i_1}\cdot \ldots \cdot a_r^{i_r},
\end{equation}
where 
\begin{equation*}
{q \choose i_1, i_2,\dots, i_r} = \frac{q!}{i_1!i_2!\cdots i_r!}
\end{equation*}
is the multinomial coefficient. Note that 
\begin{equation} \label{mcoef}
{q \choose i_1, i_2,\dots, i_r} = 0 \quad \text{ if }\  i_j<0 \text{ or } i_j>q.
\end{equation}
The sum in the multinomial formula is taken over all nonnegative integer indices $i_1, \ldots, i_r$ such that the sum of all $i_j$ is $q.$ In the case $r = 2,$ we get the binomial theorem. 
\vskip 2mm
\noindent



To give the precise statement of Theorem A, we introduce the following coefficients, which are sums of mixed products of the elementary symmetric functions of the principal curvatures, up to some multinomial coefficients. For any real $p \ne -n$,
\begin{eqnarray}\label{MT1A}
&& A^m_{p} = A^m_{p}\left(\bar{\xi}_K(u) \right) =  \\
&& \hskip -4mm  \sum_{\substack{
		i_1, \dots, i_{n-1} \geq 0 \\
		i_1 + 2i_2 + \dots + (n-  1)i_{n-1}=m}}
{\frac {n}{n+p} \choose i_1 + \dots + i_{n - 1}} {i_1 + \dots + i_{n - 1}\choose i_1,\, i_2, \, \ldots, i_{n-1}} \prod\limits_{j = 1}^{n - 1}\left\{ {n - 1\choose j}^{i_j} H_{j}^{i_j}\left(\bar{\xi}_K(u) \right)\right\}.  \nonumber 
\end{eqnarray}
For $p = -n$, we have
\begin{equation} \label{MT1A-n}
A^m_{-n}= A^m_{-n}(u) = \sum_{\substack{
		i_1, \dots, i_{2n} \geq 0\\
		i_1 + 2i_2 + \dots + 2ni_{2n}=m}} 
{\frac12 \choose i_1 + \dots + i_{2n}} {i_1 + \dots+ i_{2n} \choose i_1,\, \dots,\,i_{2n}}  \prod\limits_{q = 1}^{2n} B_q^{i_q}
\end{equation}
and
\begin{equation*}
B_q = B_q(u) = \sum\limits_{k + i = q} \left[{n - 1 \choose k}{n + 1 \choose i} \frac{H_k(\bar{\xi}_K(u))}{h_K(u)^{i}}\right].
\end{equation*}
Recall that the $L_p$ quermassintegrals are defined as
\begin{equation*} 
\mathcal{W}_{m,k} (K)= \int\limits_{S^{n - 1}} f_K(u)^{\frac n{n + p}} h_K(u)^{\frac{n(1-p)}{n+p}-k + m} A^m_p\, d\sigma(u).
\end{equation*}

\begin{theorem}\label{MT1}
	Let $K$ be a convex body in $\mathbb{R}^n$ that is $C^2_+$. Let $t \in \mathbb {R}$ be such that $ 0 \leq t  < \beta(K)$.  For all $p \in \mathbb{R}$, $p \neq -n$, 
	\begin{eqnarray} \label{p-formula}
	&& as_{p}(K + tB^n_2) = \sum\limits_{m = 0}^{\infty} \Biggl[\sum\limits_{k = m}^{\infty} {\frac{n(1-p)}{n+p} \choose k - m}\, \mathcal{W}_{m, k}\,t^k \Biggr].
	\end{eqnarray}
	
	\noindent In particular, 
	\begin{eqnarray}\label{MT1p1}
	as_{1}(K + tB^n_2) = 
	\sum\limits_{m = 0}^{\infty} \Biggl[ \int\limits_{S^{n - 1}} f_K(u)^{\frac n{n + 1}}  A^m_1\, d\sigma(u)\Biggr] \   t^m = \sum\limits_{m = 0}^{\infty} \mathcal{W}_{m,m} \,  t^m.
	\end{eqnarray}
\end{theorem}
The cases $p = 0$ and $p = \pm \infty$ are Corollary~\ref{C1} and Corollary~\ref{C2}, respectively.

\vskip 2mm
\noindent

\vskip 2mm
The next Theorem \ref{MT2} describes the case $p=-n$.
\vskip 2mm
\begin{theorem}\label{MT2}
	Let $K$ be a convex body in $\mathbb{R}^n$ that is $C^2_+$.  Let $t \in \mathbb {R}$ be such that $ 0 \leq t  < \beta(K)$. Then
	\begin{eqnarray*} \label{-n-formula}
		&& as_{-n}(K+tB^n_2) = \max _{u \in S^{n-1}} f_K(u)^\frac{1}{2} h_K(u)^\frac{n+1}{2}\sum\limits_{m = 0}^\infty A^m_{-n}t^m.
	\end{eqnarray*}
\end{theorem}
Observe that the first coefficient in the expansion is $as_{-n}(K)$.

\vskip3mm

\noindent
{\bf Remark on the polytopal case}

\noindent
When $K=P$ is a polytope, we denote by $\text{vert}\,P$ the set of its vertices and for $v \in \text{vert}\,P$,  
$$
\text{Norm} (v)= \{u\in\R^n:\, \left<u,\, z-v\right> \leq 0 \text{ for all } z\in P\}
$$ 
is the normal cone to $P$ at $v$, see \cite{SchneiderBook}.
Then the following Steiner formula for polytopes holds.

\begin{theorem}\label{MT3}
	Let $P$ be a convex polytope  in $\mathbb{R}^n$. Let $t \in \mathbb {R}$ be such that $ 0\leq t< \beta(P)$.
	\par
	\noindent
	For all $p \in \mathbb{R}$ such that

		\begin{eqnarray*} \label{p-formula polytope}
		&(i)&  p \notin [-n, 0], \quad \\
		&&as_{p}(P + t B^n_2) = \sum\limits_{m = 0}^{\infty} {\frac{n(1-p)}{n+p} \choose m}   \  \sum_{v \in \text{vert}\,P }  \, \int\limits_{u \in \text{Norm} (v)} h_P(u)^{\frac{n(1-p)}{n+p}- m} d\sigma(u)  \  t^{m + \frac{n(n-1)}{n+p}};\\
		&(ii)& p \in (-n, 0), \qquad as_p(P + tB^n_2) = \infty;\\
		&(iii)&  p = 0, \qquad \qquad \  as_0(P+tB^n_2) = n\,\vol_n(P+tB^n_2).
		\end{eqnarray*}

	\noindent
\end{theorem}
Note, that by definition \eqref{L-n}, $as_{-n}(P)$ is trivially infinite, since the curvature function of the flat part is infinite. Of course, $as_{-n}(P + tB^n_2)$ is also infinite by the same reasoning.

\subsection{Local version}
In this section, we introduce new curvature and area measures for Borel sets on the boundary of a convex body $K$ and for Borel sets on the Euclidean sphere $S^{n-1}$,   and establish local Steiner formula of the $L_p$ Brunn Minkowski theory.

When $K$ is  $C^2_+$ convex body the curvature, respectively area, measures for Borel sets $B \in \mathcal{B}(\R^n)$ and $\omega \in \mathcal{B}(S^{n-1})$ defined as
\begin{align*}
C_i(K, \beta) &= \int\limits_{\partial K \cap B} H_{n - 1 - i}(x)\, d\mathcal{H}^{n-1}(x);\\
S_i(K, \omega) &= \int\limits_{\omega} s_i\, d\sigma(u)
\end{align*}
for $i = 0, \dots, n-1$, e.g., \cite{SchneiderBook}. Note that for general convex bodies these measures replace the elementary symmetric functions of principal curvatures and the elementary symmetric functions of the principal radii of curvature. In the extreme cases $i = 0$ and $i = n-1$, we obtain $C_{n - 1}(K, B) = \mathcal{H}(\partial K \cap B)$  and $S_0 (K, \omega) = \sigma(\omega)$. If $B = \R^n$ and $\omega = S^{n-1}$, we get the classical quermassintegrals
\begin{equation*}
W_i(K) = \frac 1n C_{n - i}(K, \R^n) = \frac 1n S_{n - i}(K, S^{n-1})
\end{equation*}
which were introduced in \eqref{QMI}. 

\vskip3mm

Our approach in Theorem~\ref{MT1} leads to new curvature and area measures.  Namely, for a Borel set $B \in \mathcal{B}(\R^n)$ and for a Borel set $\omega \in \mathcal{B}(S^{n-1})$
\begin{equation*}
\mathcal{C}_{m, k}(K, B) = \int\limits_{\partial K \cap B}  \frac{H_{n-1}(x)^{\frac p{n+p}}}{\left<x, \nu(x)\right>^{\frac{n(p-1)}{n + p} + k - m}}\, A_p^m(x)  \,d\mathcal{H}^{n-1}
\end{equation*}
and 
\begin{equation*}
\mathcal{S}_{m, k}(K, \omega) = \int\limits_{\omega} s_{n-1}(u)^{\frac n{n+p}} h_K(u)^{\frac{n(1-p)}{n + p} - k + m} A^m_p (\bar{\xi}(u)) \,d\sigma(u)
\end{equation*}
for $k \geq m$ and $m \geq 0$. Taking $B = \R^n$ and $\omega = S^{n-1}$, we recover the $L_p$ quermassintegrals, i.e. 
$$\mathcal{W}_{m, k}(K) = \mathcal{C}_{m, k}(K, \R^n)\quad  \text{ and } \quad \mathcal{W}_{m, k}(K) = \mathcal{S}_{m, k}(K, S^{n-1}).$$ 
When  $m = k = 0$, $B = \R^n$ and $\omega = S^{n-1}$, we obtain the $L_p$ affine surface area, i.e. 
$$\mathcal{C}_{0, 0}(K, \R^n) = \mathcal{S}_{0, 0}(K, S^{n-1}) = as_p(K).$$

Now we can state the local Steiner formula for these newly introduced measures. 

\begin{theorem}[Local Steiner formula]
	Let $K$ be a convex body in $\R^n$ that is $C^2_+$, $B \in \mathcal{B}(\R^n)$ and $\omega \in \mathcal{B}(S^{n-1})$. Let $t$ be such that $0 \leq t < \beta(K)$. For all $p \in \R^n, \ p \ne -n$, we have
	\begin{eqnarray*} \label{local p-formula}
		&& \mathcal{C}_{0, 0}(K + tB^n_2, B) = \sum\limits_{m = 0}^{\infty} \sum\limits_{k = m}^{\infty} {\frac{n(1-p)}{n+p} \choose k - m}\, \mathcal{C}_{m, k}(K, B)\,t^k
	\end{eqnarray*}
	and
	\begin{equation*}
	\mathcal{S}_{0, 0}(K + tB^n_2, \omega)= \sum\limits_{m = 0}^{\infty} \sum\limits_{k = m}^{\infty} {\frac{n(1-p)}{n+p} \choose k - m}\, \mathcal{S}_{m, k}(K, \omega)\,t^k.
	\end{equation*}
\end{theorem}
We recover the Steiner formula~\eqref{p-formula} of $L_p$ Brunn Minkowski theory when $B = \R^n$ or $\omega = S^{n-1}$.


\section{Properties of the coefficients} \label{Coeff}
In this section, we discuss some properties of the new coefficients which appeared in Theorem~\ref{MT1}, i.e. the $L_p$ quermassintegrals,
$$
\mathcal{W}_{m, k}(K) = \int\limits_{S^{n - 1}} f_K(u)^{\frac n{n + p}} h_K(u)^{\frac{n(1-p)}{n+p}-k + m} A^m_p\, d\sigma(u)
$$
for $k \geq m$ and $m\in \mathbb{N}\cup\{0\},$ where $A_p^m$ given by~(\ref{MT1A}). We will see that the $s$-mixed $p$-affine surface areas, defined in Section~\ref{MAS}, appear as special cases of $L_p$ quermassintegrals.
As it is known that mixed affine surface areas in general are not affine invariant quantities, we cannot expect affine invariance for $L_p$ quermassintegrals either.

\subsection{Willmore energy}
Firstly, we restrict ourselves to  three-dimensional space. Recall that the Willmore energy of a compact surface $\Sigma$ in $\R^3$ is given as
$$
W_E(\Sigma) = \int_\Sigma H_1^2 \, d\mathcal{H}^{2},
$$
where $ H_1 = (k_1 + k_2)/2$ is the mean curvature. The Willmore energy naturally appears in mathematical biology and physics, and has been widely studied. In the 1960s Willmore \cite{Wi} conjectured the lower bound on the Willmore energy of a torus immersed in $\R^3$. This conjecture was proved only recently in \cite{MN}.  The choice of the exponent $2$ of the mean curvature and dimension $n = 3$ in the definition of the Willmore energy is the proper fit in the context of differential geometry as $W_E$ is invariant under conformal maps. Natural generalization of the type $\int_{\Sigma} H_1^n \, d\mathcal{H}^{n-1}$ of the Willmore energy to higher dimensional hypersurfaces  is called Willmore-Chen functional. When $n = 2$, it coincides with the Willmore energy. The  Willmore energy and Willmore-Chen functionals have been studied in \cite{ABG, Chen, LiYau} and references therein.  If one considers integrals of the type $\int_{\Sigma} H_1^\alpha \, d\mathcal{H}^{n-1}$, then they lose their conformal invariance and become much more  difficult to study. For any dimension $n$, we recover such integrals as the $L_p$ quermassintegrals in the Steiner formula \eqref{p-formula}.

\subsection{The case $p = 1$}

We start analyzing the coefficients from expansion~(\ref{MT1p1}). We observe that among them are mixed affine surface areas which will be introduced in Section~\ref{MAS}. 

If $m = 0,$ then $A_1^0 = 1$ and we have
\begin{equation*}
\mathcal{W}_{0, 0}(K)=\int\limits_{S^{n - 1}} f_K(u)^{\frac n{n + 1}} A^0_1\, d\sigma(u) = \int\limits_{S^{n - 1}} f_K(u)^{\frac n{n + 1}}\, d\sigma(u) = as_1(K).
\end{equation*}
Hence, the first term in~(\ref{MT1p1}) is the classical affine surface area of a body $K.$

If $m=l(n-1),\ l\in \mathbb{N}$, then $A_1^{l(n-1)} = {\frac n{n + 1} \choose l} H_{n-1}^l.$ Thus,
\begin{eqnarray*}
	\mathcal{W}_{l(n-1), l(n-1)}(K)&=&\int\limits_{S^{n - 1}} f_K(u)^{\frac n{n + 1}} A^{l(n-1)}_1\, d\sigma(u) = {\frac n{n + 1} \choose l} \, \int\limits_{S^{n - 1}} f_K(u)^{\frac n{n + 1} - l}\, d\sigma(u)\\
	&=& {\frac n{n + 1} \choose l} \, \int\limits_{S^{n - 1}} f_K(u)^{\frac {n - l(n+1)}{n + 1}}\, d\sigma(u) = {\frac n{n + 1} \choose l} \, as_{1,\,l(n+1)}(K).
\end{eqnarray*}
where $as_{1,\,l(n+1)}(K)$ is the $l(n+1)$-mixed 1-affine surface area of $K$.
Therefore, in~(\ref{MT1p1}) we have mixed affine surface areas as coefficients in front of powers of $t,$ which are multiples of $n - 1. $ 

\subsection{The general case}

Now we move to analyzing the coefficients appearing in~(\ref{p-formula}). We note that  again mixed affine surface areas appear.

If $m = 0$, then $A_p^0 = 1$ and we have
\begin{eqnarray*}
	\mathcal{W}_{0, k}(K)&=&\int\limits_{S^{n - 1}} f_K(u)^{\frac n{n + p}} h_K(u)^{\frac{n(1-p)}{n+p}-k} A^0_p\, d\sigma(u) = \int\limits_{S^{n - 1}} f_K(u)^{\frac n{n + p}} h_K(u)^{\frac{n(1-p)}{n+p}-k} \, d\sigma(u)\\
	&=& as_{p + \frac kn (n + p),\, -k}(K),
\end{eqnarray*}
for $k\in\mathbb{N}\cup\{0\}.$ When $k = 0,$ we get the $p$-affine surface area $as_p(K)$ of a body $K.$

If $m=l(n-1),\ l\in \mathbb{N}$, then $A_p^{l(n-1)} = {\frac n{n + p} \choose l} H_{n-1}^l.$ Thus,
\begin{eqnarray*}
	\mathcal{W}_{l(n-1), k}(K)&=&\int\limits_{S^{n - 1}} f_K(u)^{\frac n{n + p}} h_K(u)^{\frac{n(1-p)}{n+p}-k + l(n - 1)} A^{l(n - 1)}_p\, d\sigma(u) =\\
	&=&{\frac n{n + p} \choose l}  \int\limits_{S^{n - 1}} f_K(u)^{\frac n{n + p} - l} h_K(u)^{\frac{n(1-p)}{n+p}-k + l(n - 1)} \, d\sigma(u)\\
	&=& as_{\frac{np + (n + p)(k - ln)}{n - l(n + p)},\, 2nl - k}(K).
\end{eqnarray*}

The $L_p$ quermassintegrals consist of combinations of mixed products of the elementary symmetric functions of the principal curvatures $H_i$ and the support function.  Applying H\"older's inequality, we can bound those integrals from above. We present one typical example and the other cases can be dealt accordingly.

For instance, consider the case when only one symmetric function of the principal curvatures $H_j$ appears in $A^m_p$, that is, if $m = lj,\ l\in \mathbb{N}$, i.e., all $i_s = 0, \ s\ne j$ and $i_j = l$ for $1\leq j \leq n - 2,$ then $A_p^{lj} ={\frac{n}{n+p}  \choose l}  {n - 1\choose j}^l H^l_j$. Using H\"older's inequality, we get an upper bound
\begin{eqnarray*}
	\mathcal{W}_{lj, k}(K)&=&\int\limits_{S^{n - 1}} f_K(u)^{\frac n{n + p}} h_K(u)^{\frac{n(1-p)}{n+p}-k} A^{lj}_p\, d\sigma(u) \\
	&=&{\frac n{n + p} \choose l} {n - 1\choose j}^l \int\limits_{S^{n - 1}} f_K(u)^{\frac n{n + p}} h_K(u)^{\frac{n(1-p)}{n+p}-k + lj} H_j^l  \, d\sigma(u)\\
	&\leq& {\frac n{n + p} \choose l} {n - 1\choose j}^l \left( \int\limits_{S^{n - 1}} H_j^{2l} \, d\sigma(u)\right)^\frac12 as_{p + \frac{(k - lj)(n + p)}{n},\, 2lj - 2k - n }(K)^{\frac12}.
\end{eqnarray*}
Applying H\"older's inequality necessary number of times, we can obtain similar bounds for any term in the expansion~(\ref{p-formula}).

\vskip3mm

The following theorem gives the inequality for the $L_p$ quermassintegrals which is similar to the inequality for the mixed affine surface areas given in \cite[Theorem 3]{Lu} when the second body is taken to be a ball.
\begin{theorem}\label{MASIneq}
	Let $K$ be a convex body and  $i,\ j,\ k \in \R$ such that $m \leq i < j < k$. Then for a fixed $m\in\mathbb{N}\cup\{0\}$
	$$
	\mathcal{W}_{m, k}(K)^{j - i} \mathcal{W}_{m, i}(K)^{k - j} \geq \mathcal{W}_{m, j}(K)^{k - i}.
	$$
	with equality if and only if $K$ is a ball.
\end{theorem}
\begin{proof}
We use H\"older's inequality 
$$
\int\limits_{S^{n-1}} g_1(u) \cdots g_q(u)\, d\sigma(u) \leq \prod\limits_{i = 1}^q\left(\int\limits_{S^{n-1}} g_i(u)^{a_i}\,d\sigma(u)\right)^{\frac1 {a_i}}
$$
with $q = 2$ and $a_1 = \frac{k - i}{j - i},\ a_2 = \frac{k - i}{k - j}$.
We take 
\begin{align*}
g_1(u) = \left(f_p(K, u)^{\frac n{n+p}} A_p^m\right)^{\frac 1{a_1}} h_k(u)^{\frac1{a_1}(m - k)}, \qquad 
g_2(u) = \left(f_p(K, u)^{\frac n{n+p}} A_p^m\right)^{\frac 1{a_2}} h_k(u)^{\frac1{a_2}(m - i)}.
\end{align*}
Then using the definition of the $L_p$ quermassintegrals \eqref{LpQMI}, we get the desired result.

In H\"older inequality the equality happens if and only if $g_1^{1/{a_1}}$ is proportional to $g_2^{1/{a_2}}$. This leads to the condition that the support function of $K$ should be a constant, i.e. $h_K(u) = const$. Then $K$ should be a ball, which follows from the fact that the support function uniquely determines a convex body.
\end{proof}
\vskip 2mm
\par
As we already shown above that some of $\mathcal{W}_{m,k}$ are mixed affine surface areas, then they should satisfy the inequality as in Theorem \ref{MASIneq}.
\vskip3mm

\subsection{Connections to information theory}

We would like  to point out a connection between the $L_p$ quermassintegrals and information theory.  To do so we  need some background.
\par
\noindent
Let $(X, \mu)$ be a measure space  and let  $dP=p\,d\mu$ and  $dQ=q\,d\mu$ be  measures on $X$ that are  absolutely continuous with respect to the measure $\mu$. 
Then the R\'enyi divergence of order $\alpha$,  introduced by  R\'enyi \cite{Ren}  for $\alpha >0$ and $\alpha \neq 1$, is defined as 
\begin{equation}\label{renyi}
D_\alpha(P\|Q)=
\frac{1}{\alpha -1} \log \int_X p^\alpha q^{1-\alpha} d\mu,
\end{equation}
It is  the convention to put  $p^\alpha q^{1-\alpha}=0$, if $p=0$ or $q=0$, even if $\alpha <0$ and $\alpha >1$. 
The integrals 
\begin{equation}\label{HellInt}
\int_X p^\alpha q^{1-\alpha} d\mu
\end{equation}
are  also  called {\em Hellinger integrals}. See e.g.  \cite {Liese/Vajda2006} for those integrals  and  additional information.   
R\'enyi divergences and Hellinger integrals 
and their related inequalities are  important tools  in  information
theory, statistics,  probability theory, machine learning and convex geometry, see e.g.,  \cite{BarronGyorfiMeulen, CaglarWerner2014, CaglarWerner2015, CoverThomas2006,  HarremoesTopsoe,  Liese/Vajda2006, ReidWilliamson2011}.

\par 
Usually, in the literature, $\alpha \geq 0$. However, we will  also consider $\alpha <0$, provided the expressions exist.
Following the ideas of  \cite{Werner2012/1}, where R\'enyi divergences for  convex bodies $K$ were introduced, we consider 
the measure space $(\partial K, \mathcal{H}^{n-1})$ and densities on $p_K, q_K$ on $\partial K$, 
\begin{equation}\label{densities}
p_K(x)= \frac{ H_{n-1} (x)}{\langle x, \nu (x) \rangle^{n}} \, , \   \ q_K(x)=\langle x, \nu(x) \rangle .
\end{equation}
\noindent 
Then 
\begin{equation}\label{PQ}
P_K=p_K\, \mu_K \ \ \ \text{and}   \ \ \   Q_K=q_K \, \mu_K
\end{equation}
are  measures on $\partial K$ that are absolutely continuous with respect  to $\mu_{K}$. We remark that those measures are the {\em cone measures} of $K$ and $K^\circ$, respectively, see e.g. \cite{PW1}.

By the change of integration formula \eqref{Integration}, the $L_p$ quermassintegrals can be written as
\begin{eqnarray} \label{LpQMI2}
\mathcal{W}_{m,k}(K) &=& \int\limits_{S^{n - 1}} f_K(u)^{\frac n{n + p}} h_K(u)^{\frac{n(1-p)}{n+p}-k + m} A^m_p\, d\sigma(u) \nonumber \\ 
&=& \int\limits_{\partial K } H_{n-1}(x)^{\frac p{n + p}} \langle x, \nu(x) \rangle ^{\frac{n(1-p)}{n+p}-k + m} A^m_p\, d\mathcal{H}^{n-1}(x) \nonumber \\
&=&  \int\limits_{\partial K }  \left( \frac{H_{n-1}(x)}{\langle x, \nu(x) \rangle ^{n} }\right) ^{\frac {p}{n + p}}  \langle x, \nu(x) \rangle ^{1- \frac{p}{n+p}}   \  \langle x, \nu(x) \rangle  ^{m-k } A^m_p\, d\mathcal{H}^{n-1}(x).
\end{eqnarray}
Thus the $L_p$ quermassintegrals are weighted (by the weight $ \langle x, \nu(x) \rangle  ^{m-k } A^m_p$)  Hellinger integrals of the measures $P_K$ and $Q_K$, respectively, 
$$
\frac{1}{\alpha -1}  \log W_{m,k}(K) 
$$
are  weighted $\alpha$-R\'enyi divergences  with weight $ \langle x, \nu(x) \rangle  ^{m-k } A^m_p$  and $\alpha = \frac{p}{n+p}$.

\section{Mixed affine surface areas} \label{MAS}

For all $p \geq 1$ and all real $s$, the $s$-th mixed $L_p$-affine surface area of $K$ is defined in~\cite{Lu}. We use generalization of this definition to all $p \ne -n$ and all real $s$, which is given in~\cite{WY} as
\begin{equation*}
as_{p,\,s}(K) = \int\limits_{S^{n - 1}} f_p(K, u)^{\frac{n - s}{n + p}}d\sigma(u),
\end{equation*}
where $f_p(K, u) = f_K(u)h_K(u)^{1 - p}$.

\begin{theorem}\label{MT4}
	Let $K$ be a convex body in $\mathbb{R}^n$ that is $C^2_+$. Let $t \in \mathbb {R}$ be such that $ 0 \leq t  < \beta(K)$. For all $p \in \mathbb{R}$, $p \neq -n$, for all $s \in \mathbb{R}$, 
	\begin{eqnarray*} 
		&& as_{p,\,s}(K + tB^n_2) = \\
		&& \sum\limits_{m = 0}^{\infty} \Biggl[\sum\limits_{k = m}^{\infty} {\frac{n-s}{n+p}(1-p) \choose k - m} \int\limits_{S^{n - 1}} f_K(u)^{\frac {n-s}{n + p}} h_K(u)^{\frac{n-s}{n+p}(1-p)-k + m} A^m_{p,\,s} \, d\sigma(u)\Biggr]  \  t^k,
	\end{eqnarray*}
	where
	\begin{eqnarray*}
		&&A^m_{p,\,s} = A^m_{p,\,s} \left(\bar{\xi}_K(u) \right) =\\
		&&\sum_{\substack{
				i_1, \dots, i_{n-1} \geq 0 \\
				i_1 + 2i_2 + \dots + (n-  1)i_{n-1}=m}}
		{\frac {n-s}{n+p} \choose i_1 + \dots + i_{n - 1}} {i_1 + \dots + i_{n - 1}\choose i_1,\, i_2, \, \ldots, i_{n-1}} \prod\limits_{j = 1}^{n - 1}\left\{ {n - 1\choose j}^{i_j} H_{j}^{i_j}\right\} 
	\end{eqnarray*}
\end{theorem}
\par
\noindent
When $s=0$, we recover Theorem \ref{MT1} and when $p = 1$ we get the expansion for the $s$-mixed affine surface areas.

\noindent{\bf Remark on the polytopal case}

\noindent The next proposition gives the corresponding Steiner formula for the $s$-mixed $p$-affine surface area of the polytope $P$.
\vskip 2mm
\noindent
\begin{theorem}\label{MT5}
	Let $P$ be a convex polytope in $\mathbb{R}^n$. Let $t$ be such that $0\leq t< \beta(P)$.
	\par
	\noindent
	For all $s \in \mathbb{R}$, for all $p \in \mathbb{R}$ such that if $n < s$ then $p \notin (-s, -n]$ and if $n > s$ then $p \notin [-n, -s)$,
	\begin{eqnarray*} \label{ps-formula polytope}
		&& as_{p,\,s}(P + tB^n_2) = \\
		&& \sum\limits_{m=0}^{\infty} {\frac{n-s}{n+p}(1-p) \choose m} \  \sum_{v \in \text{vert P} }  \, \int\limits_{u \in \text{Norm} (v)} h_P(u)^{\frac{(n-s)(1-p)}{n+p}- m} d\sigma(u)  \  t^{m+ \frac{(n-1)(n-s)}{n+p}}.
	\end{eqnarray*}
	
\end{theorem}

\section{Proofs} \label{proofs}

We  start by showing  that Corollary \ref{C1} and \ref{C2} are a consequence of Theorem \ref{MT1}.

\subsection {Proof of Corollary \ref{C1} and \ref{C2}}

\vskip 3mm
\noindent
{\bf Proof of Corollary \ref{C1}}
\vskip 2mm
\noindent
By (\ref{0-asa}), we get for the left hand side of (\ref{p-formula}), 
$$as_{0}(K + tB^n_2) = n  \vol(K + tB^n_2).
$$
The right hand side of  (\ref{p-formula}) becomes
\begin{eqnarray*}
	&&\sum\limits_{m = 0}^{\infty} \Biggl[\sum\limits_{k = m}^{\infty} {1 \choose k - m} \  t^k \int\limits_{S^{n - 1}} f_K(u)  h_K(u)^{1 -k + m} A^m_0\, d\sigma(u)\Biggr] =\\
	&&
	\sum\limits_{m = 0}^{\infty}\Biggl[\  t^m  \  \int\limits_{S^{n - 1}} f_K(u) h_K(u) A^m_0\, d\sigma(u)  +  t^{m+1}\   \int\limits_{S^{n - 1}} f_K(u)  A^{m}_0\, d\sigma(u)  \Biggr] ,
\end{eqnarray*}
where 
\begin{eqnarray*}
	&& A^m_{0} =  \\
	&& \hskip -4mm  \sum_{\substack{
			i_1, \dots, i_{n-1} \geq 0 \\
			i_1 + 2i_2 + \dots + (n-  1)i_{n-1}=m}}
	{1 \choose i_1 + \dots + i_{n - 1}} {i_1 + \dots + i_{n - 1}\choose i_1,\, i_2, \, \ldots, i_{n-1}} \prod\limits_{j = 1}^{n - 1}\left\{ {n - 1\choose j}^{i_j} H_{j}^{i_j}\right\}.
\end{eqnarray*}
By (\ref{gcoef}), we only get a contribution for $A^m_{0}$ if either $i_1 + i_2 + \dots + i_{n-1}=0$, i.e., all $i_j=0$, or if $i_1 + i_2 + \dots + i_{n-1}=1$, which means that 
$i_j=1$ and $i_k=0$ for all $k \neq j$. Then $m=j$ as the summation is over all $i_1, \dots, i_{n-1} \geq 0$ such that 
$i_1 + 2i_2 + \dots + (n-  1)i_{n-1}=m$.  We also use (\ref{Integration}) and then we get for the right hand side of  (\ref{p-formula})
\begin{eqnarray*}
	&&\sum\limits_{m = 0}^{n-1}   {n - 1\choose m } \Biggl[ t^m  \  \int\limits_{\partial K}  \langle \nu_K(x), x \rangle H_m  \, d\mathcal{H}^{n - 1}  +  t^{m+1}\   \int\limits_{\partial K} H_m d\mathcal{H}^{n - 1}  \Biggr]  \\
	&&= \int\limits_{\partial K}  \langle \nu_K(x), x \rangle   \, d\mathcal{H}^{n - 1}(x) \\
	&&+ \sum\limits_{m = 1}^{n-1}   {n - 1\choose m } \Biggl[ t^m  \  \int\limits_{\partial K}  H_{m-1}(x)  \, d\mathcal{H}^{n - 1}(x)  +  t^{m+1}\   \int\limits_{\partial K} H_m(x) d\mathcal{H}^{n - 1} (x) \Biggr] .
\end{eqnarray*}
In the last equality we have used (\ref{mean curvature functions relation}). Collecting terms and using  the recursive identity for binomial coefficients
$$
{n \choose k} = {n - 1 \choose m}  + {n - 1 \choose m - 1}
$$
for all integers  $n, m$ such that $1 \leq m \leq n - 1$, 
and the identity (\ref{QMI}), we get for the right hand side of  (\ref{p-formula})
$$
n\vol(K) + \sum\limits_{m = 1}^{n}   {n \choose m } \    \int\limits_{\partial K}  H_{m-1}(x)  \, d\mathcal{H}^{n - 1}(x)  \  t^m =n  \sum\limits_{m = 0}^{n}   {n \choose m } \   W_m \ t^m.
$$
This shows that the classical Steiner formula is a corollary of Theorem \ref{MT1}.
\vskip 3mm
\noindent
{\bf Proof of Corollary \ref{C2}}
\vskip 2mm
\noindent
By (\ref{def:infty}) and  (\ref{dualQMI}) we get for the left hand side of (\ref{p-formula}), 
$$as_{\pm \infty }(K + tB^n_2) = n\   \vol((K + tB^n_2)^\circ) = n \  \widetilde{W}_{0} ((K + tB^n_2)^\circ).
$$
The right hand side of  (\ref{p-formula}) becomes
\begin{eqnarray*}
	&&\sum\limits_{m = 0}^{\infty} \Biggl[\sum\limits_{k = m}^{\infty} {-n \choose k - m} \  t^k \int\limits_{S^{n - 1}}  h_K(u)^{-n -k + m} A^m_{\pm \infty}\,d\sigma(u)\Biggr],
\end{eqnarray*}
where 
\begin{eqnarray*}
	&& A^m_{\pm \infty} =  \\
	&& \hskip -4mm  \sum_{\substack{
			i_1, \dots, i_{n-1} \geq 0 \\
			i_1 + 2i_2 + \dots + (n-  1)i_{n-1}=m}}
	{0 \choose i_1 + \dots + i_{n - 1}} {i_1 + \dots + i_{n - 1}\choose i_1,\, i_2, \, \ldots, i_{n-1}} \prod\limits_{j = 1}^{n - 1}\left\{ {n - 1\choose j}^{i_j} H_{j}^{i_j}\right\}.
\end{eqnarray*}
We only get a contribution for $A^m_{\pm \infty}$ if  $i_1 + i_2 + \dots + i_{n-1}=0$, i.e., if $i_j=0$ for all $j$. This means that we only get a contribution  for $m=0$. 
As $A^0_{\pm \infty} =1$, 
the right hand side of  (\ref{p-formula}) becomes
\begin{eqnarray*}
	\sum\limits_{k = 0}^{\infty}   {-n \choose k }\   t^k  \  \int\limits_{S^{n - 1}}  h_K(u)^{-n -k }  \, d\sigma(u).
\end{eqnarray*}
If $K$ is a convex body such that $0 \in \text{int}(K)$,  then 
$$
\rho_{K^\circ}(u) = \frac1{h_K(u)} \quad \quad \text{ for all } u\in S^{n - 1}. 
$$
Hence, with (\ref{dualQMI}), 
\begin{eqnarray*}
	&&\sum\limits_{k = 0}^{\infty}   {-n \choose k } \  t^k  \  \int\limits_{S^{n - 1}}  h_K(u)^{-n -k }  \, d\sigma (u) = \\
	&&\sum\limits_{k = 0}^{\infty}   {-n \choose k } \  t^k  \  \int\limits_{S^{n - 1}}  \rho_{K\circ} (u)^{n +k}  \, d\sigma (u) 
	\  =  \  n\  \sum\limits_{k = 0}^{\infty}   {-n \choose k }  \widetilde{W}_{-k} (K^\circ)  \  t^k.
\end{eqnarray*}
\vskip 3mm

\vskip 4mm
\noindent
Before we prove the general case of Theorem \ref{MT1}, it will be helpful to  treat a special case.
\vskip 2mm
Throughout the paper, we will also need the facts that are listed next.
\par
\noindent
First note that it is not difficult to see that the support function of $K+tB^n_2$ can be expressed in terms of support function of $K$ as 
\begin{equation}\label{supportfunction}
h_{K+tB^n_2}(u) = h_K(u) + t.
\end{equation}
Now we write the expression for the curvature function $f_{K + tB^n_2}(u).$ Recall that the curvature function is reciprocal of the Gauss curvature, that is,
\begin{equation*}
f_{K + tB^n_2}(u) = \frac1{H_{n-1}(\bar{\xi}_{K + tB^n_2}(u))}.
\end{equation*}
Since $\bar{\xi}_{K + tB^n_2}(u)$ is the point on $\partial (K + tB^n_2)$ that has  $u$ as unique  outer unit normal,  $\bar{\xi}_{K + tB^n_2}(u)= x + tu$,
where $x$ is this point on $\partial K$ that has $u$ as unique outer normal, i.e., $x=\bar{\xi}_{K}(u)$.  We will also use the fact that the Gauss curvature 
$
H_{n-1}(x + tu)$
is the product of  the principal curvatures $k_1^t(x + tu), \dots, k_{n - 1}^t(x + tu).$ A  well-known fact from differential geometry provides the connection between the principal curvatures $k_i^t$ of the outer parallel body $K + tB^n_2$ and principal curvatures $k_i$ of the body $K,$ namely
$$
k_i^t (x + tu) = \frac{k_i(x)}{1 + k_i(x)}, 
$$
for $x\in \partial K$ and $u$ the outer unit normal vector to $K$ at the point $x.$ Therefore, the Gauss curvature of the parallel body is
\begin{eqnarray*} \label{Gausscurvature}
	H_{n-1}(x + tu) = \prod_{i = 1}^{n - 1}k_i^t(x + tu) = \prod_{i = 1}^{n - 1} \frac{k_i(x)}{1 + tk_i(x)} = \frac{H_{n-1}(x)}{\prod_{i = 1}^{n - 1} \left(1+ tk_i(x)\right)}.
\end{eqnarray*}
Since $u$ is a unit outward normal vector at $x$ to $K$ (and  also a unit outer  normal vector at $x + tu$ to $K + tB^n_2$), we derive an expression for the curvature function $f_{K + tB^n_2}(u)$
\begin{eqnarray} \label{curvature function of parallel body}
f_{K + tB^n_2}(u) &=& f_K(u) \prod_{i = 1}^{n - 1} \Big(1+ tk_i\big(\bar{\xi}_K(u)\big)\Big) = f_K(u) \sum_{k = 0}^{n - 1} {n - 1 \choose k} H_k(\bar{\xi}_K(u)) t^k \nonumber\\
&=& f_K(u) \left(1 +  \sum_{k = 1}^{n - 1} {n - 1 \choose k} H_k(\bar{\xi}_K(u)) t^k\right).
\end{eqnarray}

\subsection{The case when $\frac{n}{n+p}$ is a natural number} \label{special case}

We consider case when $\frac{n}{n+p}$ is a natural number, that is
$$
\frac{n}{n+p} = l \quad \text{ where } l\in\mathbb{N},
$$
or equivalently,  
$$
p = -\frac{n(l-1)}{l},  \quad \text{ for } l\in\mathbb{N}.
$$
Then $(1 - p)\frac{n}{n+p} = l +n(l - 1) \in \mathbb{N},$ since $l \in \mathbb{N}.$ 

Then, by (\ref{pasa}), (\ref{supportfunction}) and (\ref{curvature function of parallel body}), 
\begin{eqnarray*}
	&&as_{p}(K+ tB^n_2) =  
	\int\limits_{S^{n - 1}}f_K^l(u)\left(1 + \sum\limits_{k = 1}^{n - 1} {n-1 \choose k}H_kt^k\right)^l \left(h_k + t\right)^{l+n(n-1)} d\sigma(u)\\
	&=& \int\limits_{S^{n - 1}}f_K^l(u)\Biggl[{l \choose 0}+{l \choose 1} \sum\limits_{k = 1}^{n - 1} {n-1 \choose k}H_kt^k + {l\choose 2} \left(\sum\limits_{k = 1}^{n - 1} {n-1 \choose k}H_kt^k\right)^2 + \ldots\\
	&+& {l\choose l} \left(\sum\limits_{k = 1}^{n - 1} {n-1 \choose k}H_kt^k\right)^l \Biggr]\cdot\\
	&&\Biggl[{l + n(l - 1)\choose 0}h^{l + n(l - 1)}_K + {l + n(l - 1)\choose 1}h^{l + n(l - 1) - 1}_Kt + \ldots + {l + n(l - 1)\choose l + n(l - 1)}t^{l + n(l-1)} \Biggr] d\sigma(u),
\end{eqnarray*}
where we used Taylor series expansion of the curvature term and  support function term. Note that $\frac t{h_K(u)} \leq \frac{t}{\beta(K)} < 1.$
Hence the binomial series for the support function uniformly converges on $S^{n - 1}.$ Since $K$ is $C^2_+$, all curvature expressions are bounded above and strictly positive, independently of $x$.


Expanding the summations, we can write the general term corresponding to the power $t^m$ as
\begin{eqnarray*}
	& &\sum_{\substack{
			i_1, \dots, i_{n-1} \geq 0\\
			i_1 + 2i_2 + \dots + (n-  1)i_{n-1}=m}}
	{l \choose i_1 + \dots + i_{n - 1}} {i_1 + \dots + i_{n - 1}\choose i_1,\, i_2, \, \ldots, i_{n-1}} {n - 1\choose 1}^{i_1}\cdot \ldots \cdot {n - 1 \choose n - 1}^{i_{n-1}} H_1^{i_1}\cdot \ldots \cdot H_{n - 1}^{i_{n - 1}} \\
	&&=\sum_{\substack{
			i_1, \dots, i_{n-1} \geq 0\\
			i_1 + 2i_2 + \dots + (n-  1)i_{n-1}=m}}
	{l \choose i_1 + \dots + i_{n - 1}} {i_1 + \dots + i_{n - 1}\choose i_1,\, i_2, \, \ldots, i_{n-1}} \prod\limits_{j = 1}^{n - 1} {n - 1\choose j}^{i_1} H_{j}^{i_j}. 
\end{eqnarray*}
The multinomial coefficients here are coming from the multinomial formula~(\ref{multiformula}).


Therefore, the following gives the general formula for the $L_p$ affine surface area
\begin{eqnarray*}
	&&as_{p}(K + tB^n_2) =\\
	&&\sum\limits_{m = 0}^{l(n-1)} \Biggl[\sum\limits_{k = m}^{l+n(l - 1) + m} {l+n(l - 1) \choose k - m}t^k \int\limits_{S^{n - 1}} f^l_K(u) h_K^{l+n(l - 1)-k + m}\cdot\\
	&&\left(\sum_{\substack{
			i_1, \dots, i_{n-1} \geq 0\\
			i_1 + 2i_2 + \dots + (n-  1)i_{n-1}=m}} {l \choose i_1 + \dots + i_{n - 1}} {i_1 + \dots + i_{n - 1}\choose i_1,\, i_2, \, \ldots, i_{n-1}} \prod\limits_{j = 1}^{n - 1}\left\{ {n - 1\choose j}^{i_j} H_{j}^{i_j}\right\} \right) d\sigma(u)\Biggr].
\end{eqnarray*}
The parameter $m$ determines the number of sums inside and changes between 0 and $l(n - 1),$ since the highest power of $t$ is $l + n(l - 1) + l(n - 1).$

For convenience, we denote the part of the expression under the integral as $A^m_{p,\,s} = A^m_{p,\,s}\left(\bar{\xi}_K(u) \right)$:
\begin{equation}\label{apm}
A^m_{p,\,s} = \sum_{\substack{
		i_1, \dots, i_{n-1} \geq 0 \\
		i_1 + 2i_2 + \dots + (n-  1)i_{n-1}=m}}
{\frac {n-s}{n+p} \choose i_1 + \dots + i_{n - 1}} {i_1 + \dots + i_{n - 1}\choose i_1,\, i_2, \, \ldots, i_{n-1}} \prod\limits_{j = 1}^{n - 1}\left\{ {n - 1\choose j}^{i_j} H_{j}^{i_j}\right\}
\end{equation}
for any real $p\ne -n$ and $s.$ Note that in the current case $p = \frac{n(1-l)}{l}, \ s=0$  and the expression $\frac n{n + p} = l.$ Coefficients $A^m_p = A^m_{p,\,0}$ represent a sum of mixed products of symmetric functions of the principal curvatures $H_i$ (\ref{ESFPC}) with corresponding multinomial coefficients.

Thus,
\begin{equation}\label{special p-formula}
as_{\frac{n(1-l)}{l}}(K+tB^n_2) = \sum\limits_{m = 0}^{l(n-1)} \left[\sum\limits_{k = m}^{l+n(l - 1) + m} {l+n(l - 1) \choose k - m}t^k \int\limits_{S^{n - 1}} f^l_K(u) h_K^{l+n(l - 1)-k + m}A^m_{\frac{n(1-l)}{l}} \, d\sigma(u)\right]
\end{equation}
for $l\in \mathbb{N}.$

\subsection{The case of real $p\ne -n$}
The $L_p$ affine surface area of $K+tB_2^n$	is given by~(\ref{pasa}).
Using relation~\eqref{curvature function of parallel body} for the curvature function of the parallel body, we can rewrite it in the following form
\begin{equation*}
as_p(K+tB_2^n) =\int\limits_{S^{n - 1}} f_K(u)^{\frac n{n + p}}\left(1 +  \sum_{k = 1}^{n - 1} {n - 1 \choose k} H_k(\bar{\xi}_K(u)) t^k\right)^{\frac{n}{n + p}} \left(h_K(u) + t\right)^{\frac {n(1-p)}{n + p}} d\sigma(u).
\end{equation*}	
Using Taylor series expansions of the curvature and support functions terms, we have
\begin{eqnarray*}
	as_p(K+tB_2^n) =\int\limits_{S^{n - 1}} f_K(u)^{\frac n{n + p}}\sum\limits_{i = 0}^\infty \left\{ {{\frac{n}{n + p}} \choose i} \left[\sum_{k = 1}^{n - 1} {n - 1 \choose k} H_k(\bar{\xi}_K(u)) t^k\right]^i \right\} \cdot\\
	\cdot \sum\limits_{j = 0}^\infty {\frac {n(1-p)}{n + p} \choose j} h_K(u)^{\frac {n(1-p)}{n + p} - j}t^j\, d\sigma(u).
\end{eqnarray*}	

Now in analogy with the procedure used in Section~\ref{special case}, we have the general formula

\begin{align*}
as_{p}&(K + tB^n_2) =\\
&\sum\limits_{m = 0}^{\infty} \Biggl[\sum\limits_{k = m}^{\infty} {\frac{n(1-p)}{n+p} \choose k - m}t^k \int\limits_{S^{n - 1}} f_K(u)^{\frac n{n + p}} h_K(u)^{\frac{n(1-p)}{n+p}-k + m}\cdot\\
&\left(\sum_{\substack{
		i_1, \dots, i_{n-1} \geq 0\\
		i_1 + 2i_2 + \dots + (n-  1)i_{n-1}=m}} {\frac{n}{n+p} \choose i_1 + \dots + i_{n - 1}} {i_1 + \dots + i_{n - 1}\choose i_1,\, i_2, \, \ldots, i_{n-1}} \prod\limits_{j = 1}^{n - 1}\left\{ {n - 1\choose j}^{i_j} H_{j}^{i_j}\right\} \right) d\sigma(u)\Biggr].
\end{align*}
Using the notation $A^m_p$ defined in~(\ref{apm}), we can write the expression above in a more compact way:
\begin{align*}
as_{p}(K + tB^n_2) =\sum\limits_{m = 0}^{\infty} \Biggl[\sum\limits_{k = m}^{\infty} {\frac{n(1-p)}{n+p} \choose k - m}t^k \int\limits_{S^{n - 1}} f_K(u)^{\frac n{n + p}} h_K(u)^{\frac{n(1-p)}{n+p}-k + m} A^m_p\, d\sigma(u)\Biggr].
\end{align*}
We note that the first coefficient in this expansion represents the $L_p$ affine surface area $as_p(K)$ of a body~$K.$

Similarly, an expansion for the $s$-th mixed $p$-affine surface area $as_{p,\,s}(K~+~tB^n_2)$ is obtained for all real $s.$ In fact, for a convex body $K$ and $ 0 \leq t  < \beta(K)$, one has 
\begin{align*}
as_{p,\,s}(K + tB^n_2) =\sum\limits_{m = 0}^{\infty} \Biggl[\sum\limits_{k = m}^{\infty} {\frac{n-s}{n+p}(1-p) \choose k - m}t^k \int\limits_{S^{n - 1}} f_K(u)^{\frac {n-s}{n + p}} h_K(u)^{\frac{n-s}{n+p}(1-p)-k + m} A^m_{p,\,s}\, d\sigma(u)\Biggr].
\end{align*}
Note that the first coefficient in this expansion gives the $s$-mixed $p$-affine affine surface area $as_{p,\,s}(K)$ of a body $K.$ If $p = 1,$ this gives the expansion for the $s$-th mixed affine surface area $as_{1,\,s}(K+tB^n_2)$ for any real $s.$

This finishes the proof of Theorem~\ref{MT1} and Theorem~\ref{MT4}.

\subsection{Proof of Theorem \ref{MT2}} \label{proofMT2}

The $L_{-n}$ affine surface area of $K + tB^n_2$ is given by
\begin{equation*}
as_{-n}(K+tB^n_2) = \max _{u \in S^{n-1}} f_{K+tB^n_2}(u)^\frac{1}{2} h_{K+tB^n_2}(u)^\frac{n+1}{2}.
\end{equation*}
Using relations (\ref{curvature function of parallel body}) and (\ref{supportfunction}), we can rewrite it as
\begin{eqnarray*}
	& &as_{-n}(K+tB^n_2) =\\
	&=& \max _{u \in S^{n-1}} f_K(u)^\frac{1}{2} h_K(u)^\frac{n+1}{2} \left(\left[1 +  \sum_{k = 1}^{n - 1} {n - 1 \choose k} H_k(\bar{\xi}_K(u)) t^k\right](h_K(u) + t)^{n+1}\right)^\frac12\\
	&=& \max _{u \in S^{n-1}} f_K(u)^\frac{1}{2} h_K(u)^\frac{n+1}{2} \left(1 + \sum\limits_{j = 1}^{2n}\,\sum\limits_{k + i = j} \left[{n - 1 \choose k}{n + 1 \choose i} \frac{H_k(\bar{\xi}_K(u))}{h_K(u)^{i}}\right] t^j \right)^\frac12.
\end{eqnarray*}
\par
\noindent
For convenience, we denote the coefficients in front of powers of $t$ as
\begin{equation*}
B_j = B_j(u) = \sum\limits_{k + i = j} \left[{n - 1 \choose k}{n + 1 \choose i} \frac{H_k(\bar{\xi}_K(u))}{h_K(u)^{i}}\right], \qquad  0 \leq j \leq 2n.
\end{equation*}
\par
\noindent
Using Taylor series expansion, we obtain
\begin{eqnarray*}
	&&as_{-n}(K+tB^n_2) = \max _{u \in S^{n-1}} f_K(u)^\frac{1}{2} h_K(u)^\frac{n+1}{2} \left(1 + \sum\limits_{j = 1}^{2n} B_j t^j\right)^\frac12\\
	&&= \max _{u \in S^{n-1}} f_K(u)^\frac{1}{2} h_K(u)^\frac{n+1}{2} \sum\limits_{m = 0}^\infty\left[ \sum_{\substack{
			i_1, \dots, i_{2n} \geq 0\\
			i_1 + 2i_2 + \dots + 2ni_{2n}=m}} 
	{\frac12 \choose i_1 + \dots + i_{2n}} {i_1 + \dots+ i_{2n} \choose i_1,\, \dots,\,i_{2n}}  \prod\limits_{q = 1}^{2n} B_q^{i_q}\right] t^m.
\end{eqnarray*}
Similarly to (\ref{apm}), we can introduce coefficients $A^m_{-n}= A^m_{-n}(u)$ as
\begin{equation*}
A^m_{-n}= A^m_{-n}(u) = \sum_{\substack{
		i_1, \dots, i_{2n} \geq 0\\
		i_1 + 2i_2 + \dots + 2ni_{2n}=m}} 
{\frac12 \choose i_1 + \dots + i_{2n}} {i_1 + \dots+ i_{2n} \choose i_1,\, \dots,\,i_{2n}}  \prod\limits_{q = 1}^{2n} B_q^{i_q}.
\end{equation*}
Then
\begin{align*}
as_{-n}(K+tB^n_2) = \max _{u \in S^{n-1}} f_K(u)^\frac{1}{2} h_K(u)^\frac{n+1}{2}\sum\limits_{m = 0}^\infty A^m_{-n}t^m.
\end{align*}

\subsection{Proof of Theorem \ref{MT3}}

For $-n < p < 0$, the exponent $\frac p{n+p}$ of the Gauss curvature $H_{n-1}(x)$ in \eqref{def:paffine} is negative. Since for the set $\{y\in\partial(P+tB^n_2): \, H_{n-1}(y) = 0 \}$ Gauss curvature is equal to zero, we have that $as_p(P+tB^n_2) = \infty$.

Next, we consider the case when $p \notin \left[-n, 0\right]$. Note, that in this case the exponent $\frac p{n+p}$ of $H_{n-1}(x)$ in \eqref{def:paffine} is positive.
We  have that   $f_{P+ t B^n_2} (u) \neq 0$, only for those $u \in S^{n-1}$ for which $f_{P+ t B^n_2} (u) < \infty$ and then $f_{P+ t B^n_2} (u)^{\frac n{n+p}} = t^{\frac{n(n-1)}{n+p}}$. 
By  this, (\ref{pasa}) and (\ref{curvature function of parallel body}) 
\begin{eqnarray*}
	as_{p}(P + t B^n_2) &=& 
	\int\limits_{ \{y\in\partial(P+tB^n_2): \, H_{n-1}(y) = 0 \}} \frac{H_{n-1}(y)^{\frac p{n+p}}}{\left<y, \nu(y) \right>^{\frac{n(p-1)}{n+p}}}  d\mathcal{H}^{n-1}(y)  \\
	&+&
	\int\limits_{ \{y\in\partial(P+tB^n_2):\, H_{n-1}(y) \ne 0 \}} \frac{H_{n-1}(y)^{\frac p{n+p}}}{\left<y, \nu(y) \right>^{\frac{n(p-1)}{n+p}}} d\mathcal{H}^{n-1}(y)\\
	&=&
	\int\limits_{ \{u: f_{P+ t B^n_2} (u) < \infty \}} f_{P+tB^n_2}(u)^{\frac n{n+p}} h_{P+tB^n_2}(u)^{\frac {n(1-p)}{n + p}} d\sigma(u)\\
	&=&
	t^{\frac{n(n-1)}{n+p}} \  \int\limits_{ \{u: f_{P+ t B^n_2} (u) < \infty \}} (h_{P}(u)  +t)^{\frac {n(1-p)}{n + p}} d\sigma(u) \\
	&=& t^{\frac{n(n-1)}{n+p}} \  \int\limits_{ \{u: f_{P+ t B^n_2} (u) < \infty \}} h_{P}(u)^{\frac {n(1-p)}{n + p}} \sum_{j=0}^ \infty {\frac{n(1-p)}{n+p} \choose j} \    \frac{t^j}{ h_{P}(u)^j} \  d\sigma(u)\\
	&=&\sum_{j=0}^ \infty     {\frac{n(1-p)}{n+p} \choose j}  \  t^{j + \frac{n(n-1)}{n+p}} \  \int\limits_{ \{u: f_{P+ t B^n_2} (u) < \infty \}} h_{P}(u)^{\frac {n(1-p)}{n + p}-j}  \  d\sigma(u)\\
	&=& \sum\limits_{j = 0}^{\infty} {\frac{n(1-p)}{n+p} \choose j}   \  \sum_{v \in \text{vert}\,P }  \, \int\limits_{u \in \text{Norm} (v)} h_P(u)^{\frac{n(1-p)}{n+p}- j} d\sigma(u)  \  t^{j+ \frac{n(n-1)}{n+p}}.
\end{eqnarray*}
As by assumption $t < \beta(P)$, we have for all $u \in S^{n-1}$ that $\frac{t}{h_P(u)} \leq \frac{t}{\beta(P)} < 1$ and therefore the above infinite sum converges uniformly.
Moreover, for all $u \in S^{n-1}$, $\lambda(P) \leq h_P(u) \leq \Lambda(P)$. Thus we can interchange integration and summation, which was done in the second  last equality above.

\par
For $p = 0$, we have a different situation, namely,
\begin{eqnarray*}
	as_{0}(P + t B^n_2) &=& 
	\int\limits_{ \{y\in\partial(P+tB^n_2): \, H_{n-1}(y) = 0 \}} \frac{1}{\left<y, \nu(y) \right>^{-1}}  d\mathcal{H}^{n-1}(y)  \\
	&+&
	\int\limits_{ \{y\in\partial(P+tB^n_2):\, H_{n-1}(y) \ne 0 \}} \frac{1}{\left<y, \nu(y) \right>^{-1}} d\mathcal{H}^{n-1}(y)\\
	&=&
	\int\limits_{ \partial(P+tB^n_2)} \left<y, \nu(y)\right> d\mathcal{H}^{n-1}(y) = n\, \vol_n(P+tB^n_2).
\end{eqnarray*}

\addressT

\vskip3mm

\addressW

\end{document}